\documentclass{amsart}
\usepackage{verbatim}
\usepackage{amssymb}
\usepackage{amsmath}
\usepackage[usenames]{color}
\usepackage{graphicx}
\usepackage{amscd}
\usepackage[numbers,sort,square]{natbib}

\usepackage{enumerate}

\usepackage[colorinlistoftodos,prependcaption,color=yellow,textsize=tiny]{todonotes}

\theoremstyle{plain}
\newtheorem{theorem}{Theorem}

\newtheorem{proposition}[theorem]{Proposition}
\newtheorem{lemma}[theorem]{Lemma}
\newtheorem{corollary}[theorem]{Corollary}

\newtheorem{definition}[theorem]{Definition}

\theoremstyle{definition}

\newtheorem{remark}[theorem]{Remark}

\newtheorem{example}[theorem]{Example}
\numberwithin{equation}{section}
\numberwithin{theorem}{section}
\allowdisplaybreaks

\usepackage{hyperref}

\def\N{{\mathbb N}}

\def\R{{\mathbb R}}

\newcommand{\E}{{\mathbb E}}

\newcommand{\F}{{\mathcal F}}

\newcommand{\eps}{\varepsilon}

\newcommand{\wt}{\widetilde}
\renewcommand{\O}{\Omega}

\newcommand{\one}{{{\bf 1}}}

\newcommand{\MQ}{\textnormal{MQ}}

\newcommand{\ud}[0]{\,\mathrm{d}}

\newcommand{\vertiii}[1]{{\left\vert\kern-0.25ex\left\vert\kern-0.25ex\left\vert #1
    \right\vert\kern-0.25ex\right\vert\kern-0.25ex\right\vert}}

\begin{document}

\title[Pointwise properties of martingales in Banach function spaces]
{Pointwise properties of martingales \\ with values in Banach function spaces}

\author{Mark Veraar}
\address{Delft Institute of Applied Mathematics\\
Delft University of Technology \\ P.O. Box 5031\\ 2600 GA Delft\\The
Netherlands}
\email{M.C.Veraar@tudelft.nl}

\author{Ivan Yaroslavtsev}
\email{I.S.Yaroslavtsev@tudelft.nl}

\begin{abstract}
In this paper we consider local martingales with values in a UMD Banach function space. We prove that such martingales have a version which is a martingale field. Moreover, a new Burkholder--Davis--Gundy type inequality is obtained.
\end{abstract}

\thanks{The first named author is supported by the Vidi subsidy 639.032.427 of the Netherlands Organisation for Scientific Research (NWO)}

\keywords{local martingale, quadratic variation, UMD Banach function spaces, Burkholder-Davis-Gundy inequalities, lattice maximal function}

\subjclass[2010]{Primary: 60G44; Secondary: 60B11, 60H05, 60G48}

\maketitle

\section{Introduction}

The discrete Burkholder--Davis--Gundy inequality (see \cite[Theorem 3.2]{Burk73}) states that for any $p\in (1, \infty)$ and martingales difference sequence $(d_j)_{j=1}^n$ in $L^p(\Omega)$ one has
\begin{equation}\label{eq:BurkholderR}
\Big\|\sum_{j=1}^n d_j\Big\|_{L^p(\Omega)} \eqsim_p \Big\|\Big(\sum_{j=1}^n |d_j|^2\Big)^{1/2}\Big\|_{L^p(\Omega)}.
\end{equation}
Moreover, there is the extension to continuous-time local martingales $M$ (see \cite[Theorem 26.12]{Kal}) which states that for every $p\in [1, \infty)$,
\begin{equation}\label{eq:Mcont}
\big\|\sup_{t\in [0,\infty)}|M_t|\big\|_{L^p(\Omega)} \eqsim_{p} \big\|[M]_{\infty}^{1/2}\big\|_{L^p(\Omega)}.
\end{equation}
Here $t\mapsto [M]_t$ denotes the quadratic variation process of $M$.

In the case $X$ is a UMD Banach function space the following variant of \eqref{eq:BurkholderR} holds (see \cite[Theorem 3]{Rubio86}): for any $p\in (1, \infty)$ and martingales difference sequence $(d_j)_{j=1}^n$ in $L^p(\Omega;X)$ one has
\begin{equation}\label{eq:BurkholderX}
\Big\|\sum_{j=1}^n d_j\Big\|_{L^p(\Omega;X)} \eqsim_p \Big\|\Big(\sum_{j=1}^n |d_j|^2\Big)^{1/2}\Big\|_{L^p(\Omega;X)}.
\end{equation}
Moreover, the validity of the estimate also characterizes the UMD property.

It is a natural question whether \eqref{eq:Mcont} has a vector-valued analogue as well.
The main result of this paper states that this is indeed the case:

\begin{theorem}\label{thm:mainintro}
Let $X$ be a UMD Banach function space over a $\sigma$-finite measure space $(S, \Sigma, \mu)$. Assume that $N:\R_+\times \O\times S\to \R$ is such that $N|_{[0,t] \times \Omega \times S}$ is $\mathcal B([0,t])\otimes \mathcal F_t\otimes \Sigma$-measurable for all $t\geq 0$ and such that for almost all $s\in S$, $N(\cdot, \cdot,s)$ is a martingale with respect to $(\F_t)_{t\geq 0}$ and $N(0,\cdot,s) = 0$. Then for all $p\in (1,\infty)$,
\begin{equation}\label{eq:BDG}
\big\|\sup_{t\geq 0} |N(t,\cdot,\cdot)| \big\|_{L^p(\Omega;X)}  \eqsim_{p,X} \sup_{t\geq 0}\big\|N(t,\cdot,\cdot) \big\|_{L^p(\Omega;X)}  \eqsim_{p,X} \|[N]_{\infty}^{1/2}\|_{L^p(\Omega;X)}.
\end{equation}
where $[N]$ denotes the quadratic variation process of $N$.
\end{theorem}
By standard methods we can extend Theorem \ref{thm:mainintro} to spaces $X$ which are isomorphic to a closed subspace of a Banach function space (e.g.\ Sobolev and Besov spaces, etc.)

The two-sided estimate \eqref{eq:BDG} can for instance be used to obtain two-sided estimates for stochastic integrals for processes with values in infinite dimensions (see \cite{NVW} and \cite{VerYar}). In particular, applying it with $N(t,\cdot,s) = \int_0^t \Phi(\cdot, s) \ud W$ implies the following maximal estimate for the stochastic integral
\begin{align}
\notag \Big\|s\mapsto \sup_{t\geq 0} \Big|\int_0^{t}\Phi(\cdot, s) \ud W\Big|  \Big\|_{L^p(\Omega; X)} & \eqsim_{p,X}
\label{eq:squarefuncW} \sup_{t\geq 0} \Big\|s\mapsto \int_0^{t}\Phi(\cdot, s) \ud W\Big\|_{L^p(\Omega; X)}
\\ &  \eqsim_{p,X} \Big\|s\mapsto \Big(\int_0^{\infty}\Phi^2(t,s)\ud t\Big)^{1/2}\Big\|_{L^p(\Omega;X)},
\end{align}
where $W$ is a Brownian motion and $\Phi:\R_+\times \Omega\times S\to \R$ is a progressively measurable process such that the right-hand side of \eqref{eq:squarefuncW} is finite. The second norm equivalence was obtained in \cite{NVW}. The norm equivalence with the left-hand side is new in this generality. The case where $X$ is an $L^q$-space was recently obtained in \cite{antoniregular} using different methods.

It is worth noticing that the second equivalence of \eqref{eq:BDG} in the case of $X=L^q$ was obtained by Marinelli in \cite{marinelli2013maximal} for some range of $1<p,q<\infty$ by using an interpolation method.

The UMD property is necessary in Theorem \ref{thm:mainintro} by necessity of the UMD property in \eqref{eq:BurkholderX} and the fact that any discrete martingale can be transformed to a continuous-time one. Also in the case of continuous martingales, the UMD property is necessary in Theorem \ref{thm:mainintro}. Indeed, applying \eqref{eq:squarefuncW} with $W$ replaced by an independent Brownian motion $\wt{W}$ we obtain
\[
 \Big\|\int_0^{\infty}\Phi \ud W\Big\|_{L^p(\Omega; X)}\eqsim_{p,X} \Big\|\int_0^{\infty}\Phi \ud \widetilde W\Big\|_{L^p(\Omega; X)},
\]
for all predictable step processes $\Phi$. The latter holds implies that $X$ is a UMD Banach space (see \cite[Theorem 1]{Ga1}).

In the special case that $X = \R$ the above reduces to \eqref{eq:Mcont}.  In the proof of Theorem \ref{thm:mainintro} the UMD property is applied several times:
\begin{itemize}
\item The boundedness of the lattice maximal function (see \cite{Bour:BCP,GCMT93,Rubio86}).
\item The $X$-valued Meyer--Yoeurp decomposition of a martingale (see Lemma~\ref{lem:XUMDdecom}).
\item The square-function estimate \eqref{eq:BurkholderX} (see \cite{Rubio86}).
\end{itemize}

It remains open whether there exists a predictable expression for the right-hand side of \eqref{eq:BDG}. One would expect that one needs simply to replace $[N]$ by its predictable compensator, the {\em predictable quadratic variation} $\langle N\rangle$. Unfortunately, this does not hold true already in the scalar-valued case: if $M$ is a real-valued martingale, then
\[
\mathbb E |M|^p_t \lesssim_{p} \mathbb E \langle M\rangle^{\frac p2}_t,\;\;\; t\geq 0,\;\; p<2,
\]
\[
\mathbb E |M|^p_t \gtrsim_{p} \mathbb E \langle M\rangle^{\frac p2}_t,\;\;\; t\geq 0,\;\; p>2,
\]
where both inequalities are known not to be sharp (see \cite[p.\ 40]{Burk73}, \cite[p.\ 297]{MarRo}, and \cite{Os12}). The question of finding such a predictable right-hand side in \eqref{eq:BDG} was answered only in the case $X=L^q$ for $1<q<\infty$ by Dirsken and the second author (see \cite{DY17}). The key tool exploited there was the so-called {\em Burkholder-Rosenthal inequalities}, which are of the following form:
\begin{equation*}
 \mathbb E \|M_N\|^p \eqsim_{p,X} \vertiii{ (M_n)_{0\leq n \leq N} }_{p,X}^p,
\end{equation*}
where $(M_n)_{0\leq n \leq N}$ is an $X$-valued martingale, $\vertiii{\cdot}_{p,X}$ is a certain norm defined on the space of $X$-valued $L^p$-martingales which depends only on {\em predictable moments} of the corresponding martingale. Therefore using approach of \cite{DY17} one can reduce the problem of continuous-time martingales to discrete-time martingales. However, the Burkholder-Rosenthal inequalities are explored only in the case $X=L^q$.

Thanks to \eqref{eq:Mcont} the following natural question arises: can one generalize \eqref{eq:BDG} to the case $p=1$, i.e.\ whether
\begin{equation}\label{eq:p=1}
  \big\|\sup_{t\geq 0} |N(t,\cdot,\cdot)| \big\|_{L^1(\Omega;X)}  \eqsim_{p,X} \|[N]_{\infty}^{1/2}\|_{L^1(\Omega;X)}
\end{equation}
holds true? Unfortunately the outlined earlier techniques cannot be applied in the case $p=1$. Moreover, the obtained estimates cannot be simply extrapolated to the case $p=1$ since those contain the {\em UMD$_p$ constant}, which is known to have infinite limit as $p\to 1$. Therefore \eqref{eq:p=1} remains an open problem. Note that in the case of a continuous martingale $M$ inequalities \eqref{eq:BDG} can be extended to the case $p\in (0,1]$ due to the classical Lenglart approach (see Corollary \ref{cor:cont0<pleq1}).

\subsubsection*{Acknowledgment} The authors would like to thank the referee for helpful comments.

\section{Preliminaries}

Throughout the paper any filtration satisfies the {\em usual conditions} (see \cite[Definition 1.1.2 and 1.1.3]{JS}), unless the underlying martingale is continuous (then the corresponding filtration can be assumed general).

\smallskip

A Banach space $X$ is called a {\it UMD space} if for some (or equivalently, for all) $p \in (1,\infty)$ there exists a constant $\beta>0$ such that
for every $n \geq 1$, every martingale difference sequence $(d_j)^n_{j=1}$ in $L^p(\Omega; X)$, and every $\{-1, 1\}$-valued sequence
$(\varepsilon_j)^n_{j=1}$
we have
\[
\Bigl(\E \Bigl\| \sum^n_{j=1} \varepsilon_j d_j\Bigr\|^p\Bigr )^{\frac1p}
\leq \beta \Bigl(\E \Bigl \| \sum^n_{j=1}d_j\Bigr\|^p\Bigr )^{\frac1p}.
\]
The above class of spaces was extensively studied by Burkholder
(see \cite{Bu1}). UMD spaces are always reflexive. Examples of UMD
space include the reflexive range of $L^q$-spaces, Besov spaces,
Sobolev, and Musielak-Orlicz spaces. Example of spaces without the
UMD property include all nonreflexive spaces, e.g. $L^1(0,1)$ and
$C([0,1])$. For details on UMD Banach spaces we refer the reader
to \cite{Burk01,HyNeVeWe16,Osekobook,Rubio86}.

The following lemma follows from \cite[Theorem 3.1]{Yarodecom}.
\begin{lemma}[Meyer-Yoeurp decomposition]\label{lem:XUMDdecom}
Let $X$ be a UMD space and $p\in (1, \infty)$. Let $M:\mathbb R_+ \times \Omega \to X$ be an $L^p$-martingale that takes values in some closed subspace $X_0$ of $X$.
Then there exists a unique decomposition $M = M^d + M^c$, where $M^c$ is continuous, $M^d$ is purely discontinuous and starts at zero, and $M^d$ and $M^c$ are $L^p$-martingales with values in $X_0\subseteq X$. Moreover, the following norm estimates hold for every $t\in [0,\infty)$,
\begin{equation}\label{eq:MYdecLpineq}
 \begin{split}
  \|M^d(t)\|_{L^p(\Omega;X)} \leq \beta_{p,X} \|M(t)\|_{L^p(\Omega;X)},\\
\|M^c(t)\|_{L^p(\Omega;X)} \leq \beta_{p,X} \|M(t)\|_{L^p(\Omega;X)}.
 \end{split}
\end{equation}
Furthermore, if $A^{p, d}_X$ and $A^{p,c}_X$ are the corresponding linear operators that map $M$ to $M^d$ and $M^c$ respectively, then
\[
 A^{p, d}_X = A^{p, d}_{\mathbb R} \otimes \textnormal{Id}_X,
\]
\[
 A^{c, d}_X = A^{c, d}_{\mathbb R}\otimes \textnormal{Id}_X.
\]
\end{lemma}

Recall that for a given measure space $(S,\Sigma,\mu)$, the linear
space of all real-valued measurable functions is denoted by
$L^0(S)$.

\begin{definition}\label{def:Bfs}
Let $(S,\Sigma,\mu)$ be a measure space. Let $n:L^0(S)\to [0,\infty]$ be a function which satisfies the following properties:
  \begin{enumerate}[(i)]
\item $n(x) = 0$ if and only if
$x=0$,
\item for all $x,y\in L^0(S)$ and $\lambda\in \R$, $n(\lambda x) = |\lambda| n(x)$ and $n(x+y)\leq n(x)+n(y)$,
    \item if $x \in L^0(S)$, $y \in L^0(S)$, and $|x| \leq |y|$, then $n(x) \leq n(y)$,
    \item if $0 \leq x_n \uparrow x$ with $(x_n)_{n=1}^\infty$ a sequence in $L^0(S)$ and $x \in L^0(S)$, then  $n(x) = \sup_{n \in \N}n(x_n)$.
  \end{enumerate}
Let $X$ denote the space of all $x\in L^0(S)$ for which
$\|x\|:=n(x)<\infty$. Then $X$ is called the {\em normed function
space associated to $n$}. It is called a {\em Banach function
space} when $(X,\|\cdot\|_X)$ is complete.
\end{definition}

 We refer the reader to
\cite[Chapter 15]{Zaa67} for details on Banach function spaces.

\begin{remark}\label{rem:continmeasure}
Let $X$ be a Banach function space over a measure space
$(S,\Sigma,\mu)$. Then $X$  is continuously embedded into $L^0(S)$
endowed with the topology of convergence in measure on sets of
finite measure. Indeed, assume $x_n\to x$ in $X$ and let $A \in\Sigma$ be of
finite measure. We claim that $\one_A x_n\to \one_A x$ in measure. For this it suffices to show that every subsequence of $(x_n)_{n\geq 1}$ has a further subsequence which convergences a.e.\ to $x$. 
Let $(x_{n_k})_{k\geq 1}$ be a subsequence. Choose a subsubsequence $(\one_A x_{n_{k_\ell}})_{\ell\geq 1} =: (y_\ell)_{\ell\geq 1}$ such that $\sum_{\ell=1}^{\infty} \|y_\ell - x\| < \infty$.
Then by \cite[Exercise 64.1]{Zaa67} $\sum_{\ell=1}^{\infty} |y_\ell - x|$ converges in $X$. In particular, $\sum_{\ell=1}^{\infty} |y_\ell - x|<\infty$ a.e. Therefore, $y_{\ell}\to x$ a.e. as desired.
\end{remark}

Given a Banach function space $X$ over a measure space $S$ and
Banach space $E$, let $X(E)$ denote the space of all strongly
measurable functions $f:S\to E$ with $\|f\|_{X(E)} :=
\big\|s\mapsto \|f(s)\|_E\big\|_X \in X$. The space $X(E)$ becomes
a Banach space when equipped with the norm $\|f\|_{X(E)}$.

A Banach function space has the UMD property if and only if
\eqref{eq:BurkholderX} holds for some (or equivalently, for all)
$p\in (1, \infty)$ (see \cite{Rubio86}). A
broad class of Banach function spaces with UMD is given by the reflexive Lorentz--Zygmund spaces (see \cite{Cobos86}) and the reflexive Musielak--Orlicz spaces (see \cite{LVY18}).

\begin{definition}
$N:\mathbb R_+ \times \Omega \times S \to \mathbb R$ is called a (continuous) (local) martingale field if $N|_{[0,t]\times \Omega \times S}$ is $\mathcal B([0,t])\otimes\mathcal F_t \otimes  \Sigma$-measurable for all $t\geq 0$  and $N(\cdot, \cdot,s)$ is a (continuous) (local) martingale with respect to $(\F_t)_{t\geq 0}$ for almost all $s\in S$.
\end{definition}

Let $X$ be a Banach space, $I \subset \mathbb R$ be a closed interval (perhaps, infinite). A function $f:I \to X$ is called {\it c\`adl\`ag} (an acronym for the French phrase ``continue \`a droite, limite \`a gauche'') if $f$ is right continuous and has limits from the left-hand side. We define a {\it Skorohod space} $\mathcal D(I; X)$ as a linear space consisting of all c\`adl\`ag functions $f:I \to X$. We denote the linear space of all bounded c\`adl\`ag functions $f:I \to X$ by $\mathcal D_b(I;X)$.

\begin{lemma}
 $\mathcal D_b(I;X)$ equipped with the norm $\|\cdot\|_{\infty}$ is a Banach space.
\end{lemma}

\begin{proof}
 The proof is analogous to the proof of the same statement for continuous functions.
\end{proof}

Let $X$ be a Banach space, $\tau$ be a stopping time, $V:\mathbb R_+ \times \Omega \to X$ be a c\`adl\`ag process. Then we define $\Delta V_{\tau}:\Omega \to X$ as follows
\[
 \Delta V_{\tau} := V_{\tau} - \lim_{\eps\to 0} V_{(\tau-\eps)\vee 0}.
\]

\section{Lattice Doob's maximal inequality}

Doob's maximal $L^p$-inequality immediately implies that for martingale fields
\[
\big\|\sup_{t\geq 0} \|N(t,\cdot)\|_X\big\|_{L^p(\Omega)} \leq \frac{p}{p-1}\sup_{t\geq 0} \|N(t)\|_{L^p(\Omega;X)},\;\;\; 1<p<\infty.
\]
In the next lemma we prove a stronger version of Doob's maximal $L^p$-inequality. As a consequence in Theorem \ref{thm:DoobLp} we will obtain the same result in a more general setting.

\begin{lemma}\label{lem:estUMDrubio}
Let $X$ be a UMD Banach function space and let $p\in (1, \infty)$. Let $N$ be a c\`adl\`ag martingale field with values in a finite dimensional subspace of $X$. Then  for all $T>0$,
\[
\big\|\sup_{t\in [0,T]}|N(t,\cdot)|\big\|_{L^p(\Omega;X)} \eqsim_{p,X} \sup_{t\in [0,T]} \|N(t)\|_{L^p(\Omega;X)}
\]
whenever one of the expression is finite.
\end{lemma}

\begin{proof}
Clearly, the left-hand side dominates the right-hand side. Therefore, we can assume  the right-hand side is finite and in this case we have
$$
\|N(T)\|_{L^p(\Omega;X)} =  \sup_{t\in [0,T]} \|N(t)\|_{L^p(\Omega;X)}<\infty.
$$
Since $N$ takes values in a finite dimensional subspace it follows from Doob's $L^p$-inequality (applied coordinatewise) that the left-hand side is finite.

Since $N$ is a c\`adl\`ag martingale field and by Definition
\ref{def:Bfs}$(iv)$ we have that
\[
\lim_{n\to \infty}\big\|\sup_{0\leq j\leq n}|N(jT/n,\cdot)|\big\|_{L^p(\Omega;X)} = \big\|\sup_{t\in [0,T]}|N(t,\cdot)|\big\|_{L^p(\Omega;X)}.
\]
Set $M_j = N_{jT/n}$ for $j\in \{0,\ldots, n\}$ and $M_{j} = M_n$ for $j> n$. It remains to prove
\[\big\|\sup_{0\leq j\leq n}|M_j(\cdot)|\big\|_{L^p(\Omega;X)} \leq C_{p,X} \|M_n\|_{L^p(\Omega;X)}.\]
If $(M_j)_{j=0}^n$ is a Paley--Walsh martingale (see \cite[Definition 3.1.8 and Proposition 3.1.10]{HyNeVeWe16}), this estimate follows from the boundedness of the dyadic lattice maximal operator \cite[pp.\ 199--200 and Theorem 3]{Rubio86}. In the general case one can replace $\Omega$ by a divisible probability space and approximate $(M_j)$ by Paley-Walsh martingales in a similar way as in \cite[Corollary 3.6.7]{HyNeVeWe16}.
\end{proof}

\begin{theorem}[Doob's maximal $L^p$-inequality]\label{thm:DoobLp}
Let $X$ be a UMD Banach function space over a $\sigma$-finite measure space and let $p\in (1, \infty)$. Let $M:\mathbb R_+\times\Omega\to X$ be a martingale such that
\begin{enumerate}
\item for all $t\geq 0$, $M(t)\in L^p(\Omega;X)$;
\item for a.a $\omega\in \O$, $M(\cdot,\omega)$ is in $\mathcal D([0,\infty);X)$.
\end{enumerate}
Then there exists a martingale field $N\in L^p(\Omega; X(\mathcal
D_b([0,\infty))))$ such that for a.a.\ $\omega\in \O$, all $t\geq
0$ and a.a.\ $s\in  S$, $N(t,\omega,s) = M(t,\omega)(s)$ and
\begin{equation}\label{eq:normestNM}
\big\|\sup_{t\geq 0}|N(t,\cdot)|\big\|_{L^p(\Omega;X)} \eqsim_{p,X} \sup_{t\geq 0}\|M(t,\cdot)\|_{L^p(\Omega;X)}.
\end{equation}
Moreover, if $M$ is continuous, then $N$ can be chosen to be continuous as well.
\end{theorem}

\begin{proof}
We first consider the case where $M$ becomes constant after some time $T>0$. Then
\[\sup_{t\geq 0}\|M(t,\cdot)\|_{L^p(\Omega;X)} = \|M(T)\|_{L^p(\Omega;X)}.\]
Let $(\xi_n)_{n\geq 1}$  be simple random variables such that $\xi_n\to M(T)$ in
$L^p(\Omega;X)$. Let $M_n(t) = \E(\xi_n|\F_t)$ for $t\geq 0$. Then by Lemma \ref{lem:estUMDrubio}
\[
\big\|\sup_{t\geq 0} |N_n(t,\cdot) - N_m(t,\cdot)|\big\|_{L^p(\Omega;X)}\eqsim_{p,X} \big\||M_n(T,\cdot) - M_m(T,\cdot)|\big\|_{L^p(\Omega;X)} \to 0
\]
as $n,m\to \infty$.
Therefore, $(N_n)_{n\geq 1}$ is a Cauchy sequence and hence converges to some $N$ from the space $L^p(\Omega;X(\mathcal D_b([0,\infty))))$. Clearly, $N(t,\cdot) = M(t)$ and
\eqref{eq:normestNM} holds in the special case that $M$ becomes constant after $T>0$.

In the case $M$ is general, for each $T>0$ we can set $M^T(t) = M(t\wedge T)$. Then for each $T>0$ we obtain a martingale field $N^T$ as required. Since $N^{T_1} = N^{T_2}$ on $[0,T_1\wedge T_2]$, we can define a martingale field $N$ by setting $N(t,\cdot) = N^T(t,\cdot)$ on $[0,T]$.
Finally, we note that
\[\lim_{T\to\infty} \sup_{t\geq 0} \|M^T(t)\|_{L^p(\Omega;X)} = \sup_{t\geq 0} \|M(t)\|_{L^p(\Omega;X)}.\]
Moreover, by Definition \ref{def:Bfs}$(iv)$ we have
\[\lim_{T\to\infty} \big\|\sup_{t\geq 0}|N^T(t,\cdot)|\big\|_{L^p(\Omega;X)} = \big\|\sup_{t\geq 0}|N(t,\cdot)|\big\|_{L^p(\Omega;X)},\]
Therefore the general case of \eqref{eq:normestNM} follows by taking limits.

Now let $M$ be continuous, and let $(M_n)_{n\geq 1}$ be as before.
By the same argument as in the first part of the proof we can
assume that there exists $T>0$ such that $M_t=M_{t\wedge T}$ for
all $t\geq 0$.
 By Lemma \ref{lem:XUMDdecom} there exists a unique decomposition $M_n = M_n^c + M_n^d$ such that $M_n^d$ is purely discontinuous and starts at zero and $M_n^c$ has continuous paths a.s. Then by \eqref{eq:MYdecLpineq}
\[
\|M(T)-M_n^c(T)\|_{L^p(\Omega;X)}\leq \beta_{p,X} \|M(T)-M_n(T)\|_{L^p(\Omega;X)}\to 0.
\]
Since $M_n^c$ takes values in a finite dimensional subspace of $X$ we can define a martingale field $N_n$ by $N_n(t,\omega,s) = M_n^c(t,\omega)(s)$. Now by Lemma \ref{lem:estUMDrubio}
\[\big\|\sup_{0\leq t\leq T} |N_n(t,\cdot) - N_m(t,\cdot)|\big\|_{L^p(\Omega;X)}\eqsim_{p,X} \big\||M_n^c(T,\cdot) - M_m^c(T,\cdot)|\big\|_{L^p(\Omega;X)} \to 0.\]
Therefore, $(N_n)_{n\geq 1}$ is a Cauchy sequence and hence converges to some $N$ from the space $L^p(\Omega;X(C_b([0,\infty))))$. Analogously to the first part of the proof, $N(t,\cdot) = M(t)$ for all $t\geq 0$.
\end{proof}

\begin{remark}\label{rem:DeltaMtau(s)=DeltaN(s)tau}
 Note that due to the construction of $N$ we have that $\Delta M_{\tau}(s) = \Delta N(\cdot, s)_{\tau}$ for any stopping time $\tau$ and almost any $s\in S$. Indeed, let $(M_n)_{n\geq 1}$ and $(N_n)_{n\geq 1}$ be as in the proof of Theorem \ref{thm:DoobLp}. Then on the one hand
 \begin{align*}
  \|\Delta M_{\tau} - \Delta (M_n)_{\tau}\|_{L^p(\Omega; X)} &\leq \bigl\|\sup_{0\leq t\leq T}\|M(t)-M_n(t)\|_X\bigr\|_{L^p(\Omega)}\\
  &\eqsim_p\|M(T)-M_n(T)\|_{L^p(\Omega;X)} \to 0,\;\;\; n\to\infty.
 \end{align*}
On the other hand
 \begin{align*}
  \|\Delta N_{\tau} - \Delta (N_n)_{\tau}\|_{L^p(\Omega; X)} &\leq \bigl\|\sup_{0\leq t\leq T}|N(t)-N_n(t)|\bigr\|_{L^p(\Omega;X)}\\
  &\eqsim_{p,X}\bigl\||N(T)-N_n(T)|\bigr\|_{L^p(\Omega;X)} \to 0,\;\;\; n\to\infty.
 \end{align*}
 Since $\|M_n(t) - N_n(t,\cdot)\|_{L^p(\Omega; X)} = 0$ for all $n\geq 0$, we have that by the limiting argument $\|\Delta M_{\tau} - \Delta N_\tau(\cdot)\|_{L^p(\Omega; X)}=0$, so the desired follows from Definition \ref{def:Bfs}$(i)$.
\end{remark}

One could hope there is a more elementary approach to derive continuity of $N$ in the case $M$ is continuous: if the filtration $\widetilde {\mathbb F} := (\widetilde{\mathcal F}_t)_{t\geq 0}$ is generated by $M$, then $M(s)$ is $\widetilde{\mathbb F}$-adapted for a.e.\ $s\in S$, and one might expect that $M$ has a continuous version. Unfortunately, this is not true in general as follows from the next example.

\begin{example}
There exists a continuous martingale $M:\mathbb R_+ \times \Omega \to \mathbb R$, a filtration $\widetilde {\mathbb F} = (\widetilde{\mathcal F}_t)_{t\geq 0}$ generated by $M$ and all $\mathbb P$-null sets, and a purely discontinuous nonzero $\widetilde {\mathbb F}$-martingale $N:\mathbb R_+ \times \Omega \to \mathbb R$. Let $W:\mathbb R_+ \times \Omega \to \mathbb R$ be a Brownian motion, $L:\mathbb R_+ \times \Omega \to \mathbb R$ be a Poisson process such that $W$ and $L$ are independent. Let $\mathbb F = (\mathcal F_t)_{t\geq 0}$ be the filtration generated by $W$ and $L$. Let $\sigma$ be an $\mathbb F$-stopping time defined as follows
 \[
  \sigma = \inf\{u\geq 0:\Delta L_u \neq 0\}.
 \]
Let us define
$$
M:= \int\mathbf 1_{[0,\sigma]}\ud W = W^{\sigma}.
$$
Then $M$ is a martingale. Let  $\widetilde {\mathbb F} := (\widetilde{\mathcal F}_t)_{t\geq 0}$ be generated by $M$. Note that $\widetilde {\mathcal F}_t \subset \mathcal F_t$ for any $t\geq 0$. Define a random variable
 \[
  \tau=\inf\{t\geq 0:\exists u\in [0,t) \,\text{such that}\, M \, \text{is a constant on}\, [u,t]\}.
 \]
Then $\tau = \sigma$ a.s. Moreover, $\tau$ is a $\widetilde {\mathbb F}$-stopping time since for each $u\geq 0$
\begin{align*}
  \mathbb P\{\tau = u\} = \mathbb P\{\sigma = u\}= \mathbb P\{\Delta L^{\sigma}_u \neq 1\} \leq \mathbb P\{\Delta L_u \neq 1\} = 0,
\end{align*}
and hence
\[
 \{\tau\leq u\} =\{\tau<u\} \cup \{\tau=u\}\subset \widetilde{ \mathcal F}_u.
\]
Therefore $N:\mathbb R_+ \times \Omega \to \mathbb R$ defined by
$$
N_t:= \mathbf 1_{[\tau,\infty)}(t) -t\wedge \tau\;\;\;\; t\geq 0,
$$
is an $\widetilde{\mathbb F}$-martingale since it is $\widetilde{\mathbb F}$-measurable and since $N_t = (L_{t}-t)^{\sigma}$ a.s. for each $t\geq 0$, hence for each $u\in [0, t]$
\[
 \mathbb E (N_t|\widetilde {\mathcal F}_u) = \mathbb E (\mathbb E (N_t|{\mathcal F}_u)|\widetilde {\mathcal F}_u)
 =\mathbb E (\mathbb E ((L_{t}-t)^{\sigma}|{\mathcal F}_u)|\widetilde {\mathcal F}_u) = (L_{u}-u)^{\sigma} = N_u
\]
due to the fact that $t\mapsto L_t-t$ is an $\widetilde{\mathbb F}$-measurable $\mathbb F$-martingale (see \cite[Problem 1.3.4]{KS}). But $(N_t)_{t\geq 0}$ is not continuous since $(L_t)_{t\geq 0}$ is not continuous.
\end{example}

\section{Main result}

Theorem \ref{thm:mainintro} will be a consequence of the following more general result.
\begin{theorem}\label{thm:main}
Let $X$ be a UMD Banach function space over a $\sigma$-finite measure space $(S, \Sigma, \mu)$ and let $p\in (1, \infty)$. Let $M:\R_+\times\O\to X$ be a local $L^p$-martingale with respect to $(\F_t)_{t\geq 0}$ and assume $M(0,\cdot) = 0$. Then there exists a mapping $N:\R_+\times \O\times S\to \R$  such that
\begin{enumerate}[$(1)$]
\item for all $t\geq0$ and a.a. $\omega\in \O$, $N(t,\omega,\cdot) = M(t,\omega)$,
\item $N$ is a local martingale field,
\item the following estimate holds
\begin{equation}\label{eq:BDG2}
\big\|\sup_{t\geq 0} |N(t,\cdot,\cdot)| \big\|_{L^p(\Omega;X)}  \eqsim_{p,X} \big\|\sup_{t\geq 0}\|M(t,\cdot)\|_X \big\|_{L^p(\Omega)} \eqsim_{p,X} \|[N]_{\infty}^{1/2}\|_{L^p(\Omega;X)}.
\end{equation}
\end{enumerate}
\end{theorem}

To prove Theorem \ref{thm:main} we first prove a completeness result.
\begin{proposition}\label{prop:completeMQ}
Let $X$ be a Banach function space over a $\sigma$-finite measure space $S$, $1\leq p<\infty$. Let
\begin{multline*}
 \MQ^p(X) := \{N:\R_+\times \Omega\times S\to\R: N \ \text{is a martingale field,}\\
 N(0, \cdot,s) = 0\; \forall s\in S, \  \text{and} \ \|N\|_{\MQ^p(X)} <\infty\},
\end{multline*}
where $\|N\|_{\MQ^p(X)} := \|[N]_{\infty}^{1/2}\|_{L^p(\Omega;X)}$. Then $(\MQ^p(X), \|\cdot\|_{\MQ^p(X)})$ is a Banach space. Moreover, if $N_n\to N$ in $\MQ^p$, then there exists a subsequence $(N_{n_k})_{k\geq 1}$ such that pointwise a.e.\ in $S$, we have $N_{n_k}\to N$ in $L^1(\Omega;\mathcal D_b([0,\infty)))$.
\end{proposition}

\begin{proof}
Let us first check that $\MQ^p(X)$ is a normed vector space. For this only the triangle inequality requires some comments. By the well-known estimate for local martingales $M, N$ (see \cite[Theorem 26.6(iii)]{Kal}) we have that a.s.
\begin{equation}\label{eq:triangleineqsqrt[M]}
 \begin{split}
  [M+N]_{t} &= [M]_t+2[M,N]_{t} +[N]_t\\
  &\leq [M]_t+2[M]^{1/2}_t[N]_{t}^{1/2} +[N]_t = \big([M]_t^{1/2}+[N]_{t}^{1/2}\big)^2,
 \end{split}
\end{equation}
Therefore, $[M+N]_{t}^{1/2} \leq [M]^{1/2}_t+[N]_{t}^{1/2}$ a.s.\ for all $t\in [0,\infty]$.

Let $(N_k)_{k\geq 1}$ be such that $\sum_{k\geq 1}
\|N_k\|_{\MQ^p(X)}<\infty$. It suffices to show that $\sum_{k\geq
1} N_k$ converges in $\MQ^p(X)$. Observe that by monotone
convergence in $\Omega$ and Jensen's inequality applied to
$\|\cdot\|_X$ for any $n > m\geq 1$ we have
\begin{equation}\label{eq:convLpxXpoint}
\begin{aligned}
\Big\|\sum_{k = m+1}^n \E [N_k]_{\infty}^{1/2}\Big\|_{X}
& = \Big\|\sum_{k = 1}^n \E [N_k]_{\infty}^{1/2} - \sum_{k = 1}^m \E [N_k]_{\infty}^{1/2}\Big\|_{X}\\
&=\Big\|\E \sum_{k = m+1}^n [N_k]_{\infty}^{1/2}\Big\|_{X} \leq
\E\Big\| \sum_{k = m+1}^n [N_k]_{\infty}^{1/2}\Big\|_{X}
\\
& =\Big\| \sum_{k = m+1}^n [N_k]_{\infty}^{1/2}\Big\|_{L^1(\Omega;X)} \leq \Big\| \sum_{k = m+1}^n [N_k]_{\infty}^{1/2}\Big\|_{L^p(\Omega;X)}\\
& \leq \sum_{k = m+1}^n\Big\|
[N_k]_{\infty}^{1/2}\Big\|_{L^p(\Omega;X)} \to 0,\;\; m,n\to
\infty,
\end{aligned}
\end{equation}
where the latter holds due to the fact that $\sum_{k \geq 1}\Big\|
[N_k]_{\infty}^{1/2}\Big\|_{L^p(\Omega;X)} < \infty$. Thus
$\sum_{k = 1}^n \E [N_k]_{\infty}^{1/2}$ converges in $X$ as $n\to
\infty$, where the corresponding limit coincides with its
pointwise limit $\sum_{k \geq 1} \E [N_k]_{\infty}^{1/2}$  by
Remark \ref{rem:continmeasure}. Therefore, since any element of
$X$ is finite a.s.\  by Definition \ref{def:Bfs}, we can find
$S_0\in \Sigma$ such that $\mu(S_0^{c}) = 0$ and pointwise in
$S_0$, we have $\sum_{k\geq 1} \E [N_k]_{\infty}^{1/2}<\infty$. 
Fix $s\in S_0$. In particular, we find that $\sum_{k\geq 1}
[N_k]_{\infty}^{1/2}$ converges in $L^1(\Omega)$. Moreover, since
by the scalar Burkholder-Davis-Gundy inequalities $\E\sup_{t\geq
0} |N_k(t,\cdot,s)| \eqsim \E[N_k(s)]_{\infty}^{1/2}$, we also
obtain that \begin{equation}\label{eq:NcontL1} N(\cdot,
s):=\sum_{k\geq 1} N_k(\cdot, s) \ \ \text{converges in} \
L^1(\Omega;\mathcal D_b([0,\infty)).
\end{equation}
Let $N(\cdot, s) = 0$ for $s\notin S_0$. Then $N$ defines a martingale field.
Moreover, by the scalar Burkholder-Davis-Gundy inequalities
\[
\lim_{m\to \infty} \Big[\sum_{k=n}^m N_k(\cdot,s)\Big]_{\infty}^{1/2} =\Big[\sum_{k=n}^\infty N_k(\cdot, s)\Big]_{\infty}^{1/2}
\]
in $L^1(\Omega)$. Therefore, by considering an a.s.\ convergent subsequence and by \eqref{eq:triangleineqsqrt[M]} we obtain
\begin{equation}\label{eq:convineqabs}
\Big[\sum_{k=n}^\infty N_k(\cdot, s)\Big]_{\infty}^{1/2}\leq \sum_{k=n}^\infty [N_k(\cdot,s)]_{\infty}^{1/2}.
\end{equation}

It remains to prove that $N\in \MQ^p(X)$ and $N = \sum_{k\geq 1} N_k$ with convergence in $\MQ^p(X)$.
Let $\varepsilon>0$. Choose $n\in \N$ such that $\sum_{k\geq n+1} \|N_k\|_{\MQ^p(X)}<\varepsilon$. It follows from \eqref{eq:convLpxXpoint} that $\E\big\| \sum_{k\geq 1} [N_k]_{\infty}^{1/2}\big\|_{X}<\infty$,
so $\sum_{k\geq 1} [N_k]_{\infty}^{1/2}$ a.s.\ converges in $X$.
Now by \eqref{eq:convineqabs}, the triangle inequality and Fatou's lemma, we obtain
\begin{align*}
\Big\|\Big[\sum_{k\geq n+1} N_k \Big]_{\infty}^{1/2}\Big\|_{L^p(\Omega; X)}
& \leq \Big\|\sum_{k=n+1}^\infty [ N_k ]_{\infty}^{1/2}\Big\|_{L^p(\Omega; X)}
\\
 & \leq \sum_{k=n+1}^\infty \Big\|[ N_k ]_{\infty}^{1/2}\Big\|_{L^p(\Omega; X)}
\\ & \leq \liminf_{m\to \infty}  \sum_{k=n+1}^m\Big\|[ N_k]_{\infty}^{1/2} \Big\|_{L^p(\Omega; X)}<\varepsilon^p.
\end{align*}
Therefore, $N\in \MQ^p(X)$ and $\| N - \sum_{k=1}^n N_k \|_{\MQ^p(X)} <\varepsilon$.

For the proof of the final assertion assume that $N_n\to N$ in $\MQ^p(X)$. Choose a subsequence $(N_{n_k})_{k\geq 1}$ such that $\|N_{n_k}- N\|_{\MQ^p(X)}\leq 2^{-k}$. Then $\sum_{k\geq 1}\|N_{n_k}- N\|_{\MQ^p(X)}<\infty$ and hence by \eqref{eq:NcontL1} we see that pointwise a.e.\ in $S$, the series $\sum_{k\geq 1} (N_{n_k}- N)$ converges in $L^1(\Omega;\mathcal D_b([0,\infty)))$. Therefore, $N_{n_k}\to N$ in $L^1(\Omega;\mathcal D_b([0,\infty);X))$ as required.
\end{proof}

For the proof of Theorem \ref{thm:main} we will need the following lemma presented in \cite[Th\'eor\`eme 2]{Dol69}.

\begin{lemma}\label{lem:Doleans}
 Let $1<p<\infty$, $M:\mathbb R_+\times\Omega \to \mathbb R$ be an $L^p$-martingales. Let $T> 0$. For each $n\geq 1$ define
 \[
  R_n := \sum_{k=1}^n \bigl|M_{\frac{Tk}{n}} - M_{\frac{T(k-1)}{n}}\bigr|^2.
 \]
Then $R_n$ converges to $[M]_T$ in $L^{p/2}$.
\end{lemma}

\begin{proof}[Proof of Theorem \ref{thm:main}]

The existence of the local martingale field $N$ together with the first estimate in \eqref{eq:BDG2} follows from Theorem \ref{thm:DoobLp}. It remains to prove
\begin{equation}\label{eq:normequivMN}
\big\|\sup_{t\geq 0}\|M(t,\cdot)\|_X \big\|_{L^p(\Omega)} \eqsim_{p,X} \|[N]_{\infty}^{1/2}\|_{L^p(\Omega;X)}.
\end{equation}
Due to Definition \ref{def:Bfs}$(iv)$ it suffices to prove the
above norm equivalence in the case $M$ and $N$ becomes constant
after some fixed time $T$.

{\em Step 1: The finite dimensional case.}
Assume that $M$ takes values in a finite dimensional subspace $Y$ of $X$ and that the right hand side of \eqref{eq:normequivMN} is finite. Then we can write $N(t,s) = M(t)(s) = \sum_{j=1}^n M_j(t) x_j(s)$, where each $M_j$ is a scalar-valued martingale with $M_j(T)\in L^p(\Omega)$ and $x_1,\ldots,x_n\in X$ form a basis of $Y$. Note that for any $c_1,\ldots,c_n\in L^p(\Omega)$ we have that
\begin{equation}\label{eq:Yfdmart}
 \Bigl\|\sum_{j=1}^n c_jx_j\Bigr\|_{L^p(\Omega; X)} \eqsim_{p,Y} \sum_{j=1}^n \|c_j\|_{L^p(\Omega)}.
\end{equation}
Fix $m\geq 1$. Then by \eqref{eq:BurkholderX} and Doob's maximal inequality
\begin{equation}\label{eq:discretetocontimeY}
\begin{split}
  \big\|\sup_{t\geq 0}\|M(t,\cdot)\|_X \big\|_{L^p(\Omega)} &\eqsim_{p} \|M(T,\cdot)\|_{L^p(\Omega;X)}\\
  &= \Bigl\|\sum_{i=1}^m M_{\frac{Ti}{m}} - M_{\frac{T(i-1)}{m}}\Bigr\|_{L^p(\Omega; X)}\\
  &\eqsim_{p,X} \Bigl\|\Bigl(\sum_{i=1}^m\bigl|M_{\frac{Ti}{m}} - M_{\frac{T(i-1)}{m}}\bigr|^2\Bigr)^{\frac{1}{2}}\Bigr\|_{L^p(\Omega; X)},
\end{split}
\end{equation}
and by \eqref{eq:Yfdmart} and Lemma \ref{lem:Doleans} the right hand side of \eqref{eq:discretetocontimeY} converges to
$$
\|[M]_{\infty}^{1/2}\|_{L^p(\Omega;X)}=\|[N]_{\infty}^{1/2}\|_{L^p(\Omega;X)}.
$$

{\em Step 2: Reduction to the case where $M$ takes values in a finite dimensional subspace of $X$.}
Let $M(T)\in L^p(\Omega;X)$. Then we can find simple functions $(\xi_n)_{n\geq 1}$ in $L^p(\Omega;X)$ such that $\xi_n\to M(T)$. Let $M_n(t) = \E(\xi_n|\mathcal{F}_t)$ for all $t\geq 0$ and $n\geq 1$, $(N_n)_{n\geq 1}$ be the corresponding martingale fields. Then each $M_n$ takes values in a finite dimensional subspace $X_n\subseteq X$, and hence by Step $1$
\[
 \big\|\sup_{t\geq 0}\|M_n(t,\cdot)-M_m(t,\cdot)\|_X \big\|_{L^p(\Omega)} \eqsim_{p,X} \|[N_n-N_m]_{\infty}^{1/2}\|_{L^p(\Omega;X)}
\]
for any $m,n\geq 1$. Therefore since $(\xi_n)_{n\geq 1}$ is Cauchy in $L^p(\Omega;X)$, $(N_n)_{n\geq 1}$ converges to some $N$ in $\MQ^p(X)$ by the first part of Proposition \ref{prop:completeMQ}.

Let us show that $N$ is the desired local martingale field. Fix $t\geq 0$. We need to show that $N(\cdot,t,\cdot) = M_t$ a.s.\ on $\Omega$. First notice that by the second part of Proposition \ref{prop:completeMQ} there exists a subsequence of $(N_n)_{n\geq 1}$ which we will denote by $(N_n)_{n\geq 1}$ as well such that $N_n(\cdot, t, \sigma) \to N(\cdot, t, \sigma)$ in $L^1(\Omega)$ for a.e.\ $\sigma\in S$. On the other hand by Jensen's inequality
\[
 \bigl\|\mathbb E |N_n(\cdot,t,\cdot) - M_t|\bigr\|_X =  \bigl\|\mathbb E |M_n(t) - M(t)|\bigr\|_X \leq \mathbb E \|M_n(t)- M(t)\|_X \to 0,\;\;\;\; n\to \infty.
\]
Hence $N_n(\cdot,t,\cdot)\to M_t$ in $X(L^1(\Omega))$, and thus by
Remark \ref{rem:continmeasure} in $L^0(S;L^1(\Omega))$. Therefore
we can find a subsequence of $(N_n)_{n\geq 1}$ (which we will
again denote by $(N_n)_{n\geq 1}$) such that
$N_n(\cdot,t,\sigma)\to M_t(\sigma)$ in $L^1(\Omega)$ for a.e.\
$\sigma\in S$ (here we use the fact that $\mu$ is
$\sigma$-finite), so $N(\cdot, t, \cdot) = M_t$ a.s.\ on $\Omega
\times S$, and consequently by Definition \ref{def:Bfs}$(iii)$,
$N(\omega, t, \cdot) = M_t(\omega)$ for a.a.\ $\omega \in \Omega$.
Thus \eqref{eq:normequivMN} follows by letting $n\to \infty$.

{\em Step 3: Reduction to the case where the left-hand side of
\eqref{eq:normequivMN} is finite.} Assume that the left-hand side
of \eqref{eq:normequivMN} is infinite, but the right-hand side is
finite. Since $M$ is a local $L^p$-martingale we can find a
sequence of stopping times $(\tau_n)_{n\geq 1}$ such that
$\tau_n\uparrow \infty$ and
$\|M^{\tau_n}_T\|_{L^p(\Omega;X)}<\infty$ for each $n\geq 1$. By
the monotone convergence theorem and Definition
\ref{def:Bfs}$(iv)$
\begin{align*}
 \|[N]_{\infty}^{1/2}\|_{L^p(\Omega;X)} &= \lim_{n\to \infty}\|[N^{\tau_n}]_{\infty}^{1/2}\|_{L^p(\Omega;X)} \eqsim_{p,X} \limsup_{n\to \infty}\|M^{\tau_n}_T\|_{L^p(\Omega;X)}\\
 & = \lim_{n\to \infty}\|M^{\tau_n}_T\|_{L^p(\Omega;X)} = \lim_{n\to \infty}\Bigl\|\sup_{0\leq t\leq T}\|M^{\tau_n}_t\|_X\Bigr\|_{L^p(\Omega)}\\
 &= \Bigl\|\sup_{0\leq t\leq T}\|M_t\|_X\Bigr\|_{L^p(\Omega)}=\infty
\end{align*}
and hence the right-hand side of \eqref{eq:normequivMN} is infinite as well.
\end{proof}

We use an extrapolation argument to extend part of Theorem \ref{thm:main} to $p\in (0,1]$ in the continuous-path case.

\begin{corollary}\label{cor:cont0<pleq1}
Let $X$ be a UMD Banach function space over a $\sigma$-finite measure space and let $p\in (0, \infty)$. Let $M$ be a continuous local martingale $M:\mathbb R_+\times\Omega\to X$ with $M(0,\cdot) = 0$. Then there exists a continuous local martingale field $N:\mathbb R_+\times \Omega\times S\to \R$ such that for a.a.\ $\omega\in \O$, all $t\geq 0$, and a.a.\ $s\in  S$, $N(t,\omega,\cdot) = M(t,\omega)(s)$ and
\begin{equation}\label{eq:normestNM0<p<infty}
\big\|\sup_{t\geq 0} \|M(t,\cdot)\|_X \big\|_{L^p(\Omega)} \eqsim_{p,X} \big\|[N]_{\infty}^{1/2}\big\|_{L^p(\Omega;X)} .
\end{equation}
\end{corollary}

\begin{proof}
By a stopping time argument we can reduce to the case where $\|M(t,\omega)\|_X$ is uniformly bounded in $t\in\mathbb R_+$ and $\omega\in \Omega$ and $M$ becomes constant after a fixed time $T$. Now the existence of $N$ follows from Theorem \ref{thm:main} and it remains to prove \eqref{eq:normestNM0<p<infty} for $p\in (0,1]$. For this we can use a classical argument due to Lenglart. Indeed, for both estimates we can apply
\cite{Lenglart} or \cite[Proposition IV.4.7]{RY} to the continuous increasing processes $Y,Z:\R_+\times\Omega\to \mathbb R_+$ given by
\begin{align*}
Y_u &=\mathbb E \sup_{t\in [0,u]} \|M(t,\cdot)\|_X,
\\  Z_u &= \|s\mapsto [N(\cdot, \cdot, s)]_{u}^{1/2}\|_X,
\end{align*}
where $q\in (1, \infty)$ is a fixed number. Then by \eqref{eq:BDG2} for any bounded stopping time $\tau$, we have
\begin{align*}
\E Y_{\tau}^q & = \sup_{t\geq 0} \|M(t\wedge \tau,\cdot)\|_X^q \eqsim_{q,X} \E \|s\mapsto [N(\cdot\wedge \tau, \cdot, s)]_{\infty}^{1/2}\|_X^q \\ & \stackrel{(*)}{=} \E \|s\mapsto [N(\cdot, \cdot, s)]_{\tau}^{1/2}\|_X^q = \E Z_\tau^q,
\end{align*}
where we used \cite[Theorem 17.5]{Kal} in $(*)$. Now \eqref{eq:normestNM0<p<infty} for $p\in (0,q)$ follows from \cite{Lenglart} or \cite[Proposition IV.4.7]{RY}.
\end{proof}

As we saw in Theorem \ref{thm:DoobLp}, continuity of $M$ implies
pointwise continuity of the corresponding martingale field $N$.
The following corollaries of Theorem \ref{thm:main}  are devoted
to proving the same type of assertions concerning pure
discontinuity, quasi-left continuity, and having accessible jumps.

Let $\tau$ be a stopping time. Then $\tau$ is called {\em predictable} if there exists a sequence of stopping times $(\tau_n)_{n\geq 1}$ such that $\tau_n<\tau$ a.s.\ on $\{\tau>0\}$ for each $n\geq 1$ and $\tau_n \nearrow\tau$ a.s. A c\`adl\`ag process $V:\mathbb R_+ \times \Omega \to X$ is called to have {\em accessible jumps} if there exists a sequence of predictable stopping times $(\tau_n)_{n\geq 1}$ such that $\{t\in \mathbb R_+:\Delta V \neq 0\} \subset \{\tau_1,\ldots,\tau_n,\ldots\}$ a.s.

\begin{corollary}\label{cor:pdmartwiajissuchinS}
 Let $X$ be a UMD function space over a measure space $(S, \Sigma, \mu)$, $1<p<\infty$, $M:\mathbb R_+ \times \Omega \to X$ be a purely discontinuous $L^p$-martingale with accessible jumps. Let $N$ be the corresponding martingale field. Then $N(\cdot, s)$ is a purely discontinuous martingale with accessible jumps for a.e.\ $s\in S$.
\end{corollary}

For the proof we will need the following lemma taken from \cite[Subsection 5.3]{DY17}.

\begin{lemma}\label{lemma:DeltaMisamart}
 Let $X$ be a Banach space, $1\leq p<\infty$, $M:\mathbb R_+ \times \Omega \to X$ be an $L^p$-martingale, $\tau$ be a~predictable stopping time. Then $(\Delta M_{\tau}\mathbf 1_{[0,t]}(\tau))_{t\geq 0}$ is an $L^p$-martingale as well.
\end{lemma}

\begin{proof}[Proof of Corollary \ref{cor:pdmartwiajissuchinS}]
Without loss of generality we can assume that there exists $T\geq 0$ such that $M_t=M_T$ for all $t\geq T$, and that $M_0=0$.
 Since $M$ has accessible jumps, there exists a sequence of predictable stopping times $(\tau_n)_{n\geq 1}$ such that a.s.
 $$
 \{t\in \mathbb R_+:\Delta M \neq 0\} \subset \{\tau_1,\ldots,\tau_n,\ldots\}.
 $$
 For each $m\geq 1$ define a process $M^m:\mathbb R_+ \times \Omega \to X$ in the following way:
 \[
  M^m(t) := \sum_{n=1}^m\Delta M_{\tau_n}\mathbf 1_{[0,t]}(\tau_n),\;\;\; t\geq 0.
 \]
Note that $M^m$ is a purely discontinuous $L^p$-martingale with accessible jumps by Lemma \ref{lemma:DeltaMisamart}. Let $N^m$ be the corresponding martingale field. Then $N^m(\cdot, s)$ is a purely discontinuous martingale with accessible jumps for almost any $s\in S$ due to Remark \ref{rem:DeltaMtau(s)=DeltaN(s)tau}. Moreover, for any $m\geq \ell\geq 1$ and any $t\geq 0$ we have that a.s.\ $[N^{m}(\cdot, s)]_t \geq [N^{\ell}(\cdot, s)]_t$. Define $F:\mathbb R_+\times \Omega \times S \to \mathbb R_+\cup\{+\infty\}$ in the following way:
$$
F(t,\cdot,s):= \lim_{m\to \infty} [N^{m}(\cdot, s)]_t,\;\;\; s\in S, t\geq 0.
$$
Note that $F(\cdot, \cdot, s)$ is a.s.\ finite for almost any $s\in S$. Indeed, by Theorem \ref{thm:main} and \cite[Theorem 4.2]{Yarodecom} we have that for any $m\geq 1$
\[
\big\|[N^m]_{\infty}^{1/2}\big\|_{L^p(\Omega;X)} \eqsim_{p,X} \|M^m(T,\cdot) \|_{L^p(\Omega;X)} \leq \beta_{p,X}\|M(T,\cdot)\|_{L^p(\Omega;X)},
\]
so by Definition \ref{def:Bfs}$(iv)$, $F(\cdot, \cdot, s)$ is
a.s.\ finite for almost any $s\in S$ and
\begin{align*}
 \big\|F_{\infty}^{1/2}\big\|_{L^p(\Omega;X)} &= \big\|F_{T}^{1/2}\big\|_{L^p(\Omega;X)} = \lim_{m\to \infty}\big\|[N^m]_{T}^{1/2}\big\|_{L^p(\Omega;X)}\\
 &\lesssim_{p,X} \limsup_{m\to \infty}\|M^m(T,\cdot) \|_{L^p(\Omega;X)} \lesssim_{p,X} \|M(T,\cdot)\|_{L^p(\Omega;X)}.
\end{align*}
Moreover, for almost any $s\in S$ we have that $F(\cdot, \cdot, s)$ is pure jump and
\[
 \{t\in \mathbb R_+:\Delta F \neq 0\} \subset \{\tau_1,\ldots,\tau_n,\ldots\}.
\]
Therefore to this end it suffices to show that $F(s)=[N(s)]$ a.s.\
on $\Omega$ for a.e.\ $s\in S$. Note that by Definition
\ref{def:Bfs}$(iv)$,
\begin{equation}\label{eq:FisqvofM^0}
 \big\|(F-[N^m])^{1/2}(\infty)\big\|_{L^p(\Omega;X)}\to 0,\;\;\; m\to \infty
\end{equation}
so by Theorem \ref{thm:main} $(M^m(T))_{m\geq 1}$ is a Cauchy sequence in $L^p(\Omega; X)$. Let $\xi$ be its limit, $M^0:\mathbb R_+ \times \Omega \to X$ be a martingale such that $M^0(t)=\mathbb E (\xi|\mathcal F_{t})$ for all $t\geq 0$. Then by \cite[Proposition 2.14]{Yarodecom} $M^0$ is purely discontinuous. Moreover, for any stopping time $\tau$ a.s.\
\[
 \Delta M^0_{\tau} = \lim_{m\to \infty} \Delta M^m_{\tau} = \lim_{m\to \infty}\Delta M_{\tau} \mathbf 1_{\{\tau_1,\ldots,\tau_m\}}(\tau) = \Delta M_{\tau},
\]
where the latter holds since the set $\{\tau_1,\ldots,\tau_n,\ldots\}$ exhausts the jump times of $M$. Therefore $M=M^0$ since both $M$ and $M^0$ are purely discontinuous with the same jumps, and hence $[N]=F$ (where $F(s) = [M^0(s)]$ by \eqref{eq:FisqvofM^0}). Consequently $N(\cdot,\cdot, s)$ is purely discontinuous with accessible jumps for almost all $s\in S$.
\end{proof}

\begin{remark}
 Note that the proof of Corollary \ref{cor:pdmartwiajissuchinS} also implies that $M^m_t \to M_t$ in $L^p(\Omega; X)$ for each $t\geq 0$.
\end{remark}

A c\`adl\`ag process $V:\mathbb R_+ \times \Omega \to X$ is called {\em quasi-left continuous} if $\Delta V_{\tau}=0$ a.s.\ for any predictable stopping time $\tau$.

\begin{corollary}\label{cor:pdqlcmartissuchinS}
 Let $X$ be a UMD function space over a measure space $(S, \Sigma, \mu)$, $1<p<\infty$, $M:\mathbb R_+ \times \Omega \to X$ be a purely discontinuous quasi-left continuous \mbox{$L^p$-mar}\-tin\-gale. Let $N$ be the corresponding martingale field. Then $N(\cdot, s)$ is a purely discontinuous quasi-left continuous martingale for a.e.\ $s\in S$.
\end{corollary}

The proof will exploit the random measure theory. Let $(J, \mathcal J)$ be a measurable space. Then a family $\mu = \{\mu(\omega; \ud t, \ud x), \omega \in \Omega\}$ of nonnegative measures on $(\mathbb R_+ \times J; \mathcal B(\mathbb R_+)\otimes \mathcal J)$ is called a {\it random measure}. A random measure $\mu$ is called {\it integer-valued} if it takes values in $\mathbb N\cup\{\infty\}$, i.e.\ for each $A \in \mathcal B(\mathbb R_+)\otimes \mathcal F\otimes \mathcal J$ one has that $\mu(A) \in \mathbb N\cup\{\infty\}$ a.s., and if $\mu(\{t\}\times J)\in \{0,1\}$ a.s.\ for all $t\geq 0$.

Let $X$ be a Banach space, $\mu$ be a random measure, $F:\mathbb R_+ \times \Omega \times J  \to X$ be such that $\int_{\mathbb R_+ \times J} \|F\| \ud \mu<\infty$ a.s. Then the integral process $((F\star \mu)_t)_{t\geq 0}$ of the form
\[
 (F\star \mu)_t := \int_{\mathbb R_+ \times J} F(s,\cdot, x)\mathbf 1_{[0,t]}(s)\mu(\cdot; \ud s, \ud x),\;\;\; t\geq 0,
\]
is a.s.\ well-defined.

Any integer-valued optional ${\mathcal P}\otimes \mathcal J$-$\sigma$-finite random measure $\mu$ has a {\em compensator}: a~unique predictable ${\mathcal P}\otimes \mathcal J$-$\sigma$-finite random measure $\nu$ such that $\mathbb E (W \star \mu)_{\infty} = \mathbb E (W \star \nu)_{\infty}$ for each ${\mathcal P}\otimes \mathcal J$-measurable real-valued nonnegative $W$ (see \cite[Theorem II.1.8]{JS}). For any optional ${\mathcal P}\otimes \mathcal J$-$\sigma$-finite measure $\mu$ we define the associated compensated random measure by $\bar{\mu} = \mu -\nu$.

Recall that $\mathcal P$ denotes the predictable $\sigma$-algebra on $\mathbb R_+ \times \Omega$ (see \cite{Kal} for details). For each $\mathcal P \otimes \mathcal J$-strongly-measurable $F:\mathbb R_+ \times \Omega \times J \to X$ such that \linebreak $\mathbb E (\|F\|\star \mu)_{\infty}< \infty$ (or, equivalently, $\mathbb E (\|F\|\star \nu)_{\infty}<\infty$, see the definition of a compensator above) we can define a process $F\star \bar{\mu}$ by $F \star \mu - F \star \nu$. Then this process is a purely discontinuous local martingale. We will omit here some technicalities for the convenience of the reader and refer the reader to \cite[Chapter II.1]{JS}, \cite[Subsection 5.4-5.5]{DY17}, and \cite{KalRM,Nov75,MarRo} for more details on random measures.

\begin{proof}[Proof of Corollary \ref{cor:pdqlcmartissuchinS}]
Without loss of generality we can assume that there exists $T\geq 0$ such that $M_t=M_T$ for all $t\geq T$, and that $M_0=0$.
Let $\mu$ be a random measure defined on $\mathbb R_+\times X$ in the following way
\[
 \mu(A\times B) = \sum_{t\geq 0} \mathbf 1_{A}(t) \mathbf 1_{B\setminus \{0\}}(\Delta M_t),
\]
where $A\subset \mathbb R_+$ is a Borel set, and $B\subset X$ is a ball. For each $k, \ell\geq 1$ we define a stopping time $\tau_{k,\ell}$ as follows
\[
 \tau_{k,\ell} = \inf\{t\in \mathbb R_+: \#\{u\in [0,t] : \|\Delta M_u\|_X\in [1/k, k]\} = \ell\}.
\]
Since $M$ has c\`adl\`ag trajectories, $\tau_{k,\ell}$ is a.s.\ well-defined and takes its values in $[0,\infty]$. Moreover, $\tau_{k,\ell}\to \infty$ for each $k\geq 1$ a.s.\ as $\ell\to \infty$, so we can find a subsequence $(\tau_{k_n,\ell_n})_{n\geq 1}$ such that $k_n\geq n$ for each $n\geq 1$ and $\inf_{m\geq n} \tau_{k_m, \ell_m}\to \infty$ a.s.\ as $n\to \infty$. Define $\tau_n = \inf_{m\geq n} \tau_{k_m, \ell_m}$ and define $M^n := (\mathbf 1_{[0,\tau_n]} \mathbf 1_{B_n})\star \bar{\mu}$, where $\bar{\mu} = \mu-\nu$ is such that $\nu$ is a compensator of $\mu$ and $B_n = \{x\in X:\|x\|\in [1/n,n]\}$. Then $M^n$ is a purely discontinuous quasi-left continuous martingale by \cite{DY17}. Moreover, a.s.\
\[
 \Delta M^n_t = \Delta M_t \mathbf 1_{[0,\tau_n]}(t) \mathbf 1_{[1/n,n]}(\|\Delta M_t\|),\;\;\;\; t\geq 0.
\]
so by \cite{Yarodecom} $M^n$ is an $L^p$-martingale (due to the {\em weak differential subordination} of purely discontinuous martingales).

 The rest of the proof is analogous to the proof of Corollary \ref{cor:pdmartwiajissuchinS} and uses the fact that $\tau_n\to \infty$ monotonically a.s.
\end{proof}

Let $X$ be a Banach space. A local martingale $M:\mathbb R_+ \times \Omega \to X$ is called to have the {\em canonical decomposition} if there exist local martingales $M^c,M^q, M^a:\mathbb R_+\times \Omega \to X$ such that $M^c$ is continuous, $M^q$ and $M^a$ are purely discontinuous, $M^q$ is quasi-left continuous, $M^a$ has accessible jumps, $M^c_0=M^q_0=0$, and $M = M^c +  M^q + M^a$. Existence of such a decomposition was first shown in the real-valued case by Yoeurp in \cite{Yoe76}, and recently such an existence was obtained in the UMD space case (see \cite{Yarodecom,Yar17GMY}).

\begin{remark}\label{rem:candecisunique}
 Note that if a local martingale $M$ has some canonical decomposition, then this decomposition is unique (see \cite{Kal,Yoe76,Yarodecom,Yar17GMY}).
\end{remark}

\begin{corollary}
 Let $X$ be a UMD Banach function space, $1<p<\infty$, $M:\mathbb R_+\times \Omega \to X$ be an $L^p$-martingale. Let $N$ be the corresponding martingale field. Let $M=M^c + M^q + M^a$ be the canonical decomposition, $N^c$, $N^q$, and $N^a$ be the corresponding martingale fields. Then $N(s)=N^c(s) + N^q(s) + N^a(s)$ is the canonical decomposition of $N(s)$ for a.e.\ $s\in S$. In particular, if $M_0=0$ a.s., then $M$ is continuous, purely discontinuous quasi-left continuous, or purely discontinuous with accessible jumps if and only if $N(s)$ is so for a.e.\ $s\in S$.
\end{corollary}

\begin{proof}
 The first part follows from Theorem \ref{thm:DoobLp}, Corollary \ref{cor:pdmartwiajissuchinS}, and Corollary \ref{cor:pdqlcmartissuchinS} and the fact that $N(s)=N^c(s) + N^q(s) + N^a(s)$ is then a canonical decomposition of a local martingale $N(s)$ which is unique due to Remark \ref{rem:candecisunique}. Let us show the second part. One direction follows from Theorem \ref{thm:DoobLp}, Corollary \ref{cor:pdmartwiajissuchinS}, and Corollary \ref{cor:pdqlcmartissuchinS}. For the other direction assume that $N(s)$ is continuous for a.e.\ $s\in S$. Let $M= M^c + M^q + M^a$ be the canonical decomposition, $N^c$, $N^q$, and $N^a$ be the corresponding martingale fields of $M^c$, $M^q$, and $M^a$. Then by the first part of the theorem and the uniqueness of the canonical decomposition (see Remark \ref{rem:candecisunique}) we have that for a.e.\ $s\in S$, $N^q(s) = N^a(s) = 0$, so $M^q = M^a = 0$, and hence $M$ is continuous. The proof for the case of pointwise purely discontinuous quasi-left continuous $N$ or pointwise purely discontinuous $N$ with accessible jumps is similar.
\end{proof}

\begin{remark}
It remains open whether the first two-sided estimate in \eqref{eq:BDG2} can be extended to $p=1$. 
Recently, in \cite{Y18BDG} the second author has extended
the second two-sided estimate in \eqref{eq:BDG2} to arbitrary UMD Banach spaces and to $p\in
[1,\infty)$. Here the quadratic variation has to be replaced by a
generalized square function. 
\end{remark}

\bibliographystyle{plain}

\begin{thebibliography}{10}

\bibitem{antoniregular}
M.~Antoni.
\newblock {\em {Regular Random Field Solutions for Stochastic Evolution
  Equations}}.
\newblock PhD thesis, 2017.

\bibitem{Bour:BCP}
J.~Bourgain.
\newblock Extension of a result of {B}enedek, {C}alder\'on and {P}anzone.
\newblock {\em Ark. Mat.}, 22(1):91--95, 1984.

\bibitem{Burk73}
D.L. Burkholder.
\newblock Distribution function inequalities for martingales.
\newblock {\em Ann. Probability}, 1:19--42, 1973.

\bibitem{Bu1}
D.L. Burkholder.
\newblock A geometrical characterization of {B}anach spaces in which martingale
  difference sequences are unconditional.
\newblock {\em Ann. Probab.}, 9(6):997--1011, 1981.

\bibitem{Burk01}
D.L. Burkholder.
\newblock Martingales and singular integrals in {B}anach spaces.
\newblock In {\em Handbook of the geometry of {B}anach spaces, {V}ol. {I}},
  pages 233--269. North-Holland, Amsterdam, 2001.

\bibitem{Cobos86}
F.~Cobos.
\newblock Some spaces in which martingale difference sequences are
  unconditional.
\newblock {\em Bull. Polish Acad. Sci. Math.}, 34(11-12):695--703 (1987), 1986.

\bibitem{DY17}
S.~Dirksen and I.S. Yaroslavtsev.
\newblock ${L}^q$-valued {B}urkholder-{R}osenthal inequalities and sharp
  estimates for stochastic integrals.
\newblock {\em arXiv:1707.00109}, 2017.

\bibitem{Dol69}
C.~Dol\'eans.
\newblock Variation quadratique des martingales continues \`a droite.
\newblock {\em Ann. Math. Statist}, 40:284--289, 1969.

\bibitem{GCMT93}
J.~Garc{\'{\i}}a-Cuerva, R.~Mac{\'{\i}}as, and J.~L. Torrea.
\newblock The {H}ardy-{L}ittlewood property of {B}anach lattices.
\newblock {\em Israel J. Math.}, 83(1-2):177--201, 1993.

\bibitem{Ga1}
D.J.H. Garling.
\newblock Brownian motion and {UMD}-spaces.
\newblock In {\em Probability and Banach spaces (Zaragoza, 1985)}, volume 1221
  of {\em Lecture Notes in Math.}, pages 36--49. Springer, Berlin, 1986.

\bibitem{HyNeVeWe16}
T.~Hyt\"onen, J.~van Neerven, M.~Veraar, and L.~Weis.
\newblock {\em Analysis in {B}anach Spaces. {V}olume {I}: {M}artingales and
  {L}ittlewood-{P}aley Theory}, volume~63 of {\em Ergebnisse der Mathematik und
  ihrer Grenzgebiete (3)}.
\newblock Springer, 2016.

\bibitem{JS}
J.~Jacod and A.N. Shiryaev.
\newblock {\em Limit theorems for stochastic processes}, volume 288 of {\em
  Grundlehren der Mathematischen Wissenschaften}.
\newblock Springer-Verlag, Berlin, second edition, 2003.

\bibitem{Kal}
O.~Kallenberg.
\newblock {\em Foundations of modern probability}.
\newblock Probability and its Applications (New York). Springer-Verlag, New
  York, second edition, 2002.

\bibitem{KalRM}
O.~Kallenberg.
\newblock {\em Random measures, theory and applications}, volume~77 of {\em
  Probability Theory and Stochastic Modelling}.
\newblock Springer, Cham, 2017.

\bibitem{KS}
I.~Karatzas and S.E. Shreve.
\newblock {\em Brownian motion and stochastic calculus}, volume 113 of {\em
  Graduate Texts in Mathematics}.
\newblock Springer-Verlag, New York, second edition, 1991.

\bibitem{Lenglart}
E.~Lenglart.
\newblock Relation de domination entre deux processus.
\newblock {\em Ann. Inst. H. Poincar\'e Sect. B (N.S.)}, 13(2):171--179, 1977.

\bibitem{LVY18}
N.~Lindemulder, M.C. Veraar, and I.S. Yaroslavtsev.
\newblock The {UMD} property for {M}usielak--{O}rlicz spaces.
\newblock In {\em Positivity and Noncommutative Analysis. Festschrift in honour
  of Ben de Pagter on the occasion of his 65th birthday}, Trends in
  Mathematics. Birkh{\"a}user, 2019.
\newblock arXiv:1810.13362, to appear.

\bibitem{marinelli2013maximal}
C.~Marinelli.
\newblock On maximal inequalities for purely discontinuous {$L_q$}-valued
  martingales.
\newblock {\em arXiv:1311.7120}, 2013.

\bibitem{MarRo}
C.~Marinelli and M.~R{\"o}ckner.
\newblock On maximal inequalities for purely discontinuous martingales in
  infinite dimensions.
\newblock In {\em S\'eminaire de {P}robabilit\'es {XLVI}}, volume 2123 of {\em
  Lecture Notes in Math.}, pages 293--315. Springer, Cham, 2014.

\bibitem{NVW}
J.M.A.M.~van Neerven, M.C. Veraar, and L.W. Weis.
\newblock Stochastic integration in {UMD} {B}anach spaces.
\newblock {\em Ann. Probab.}, 35(4):1438--1478, 2007.

\bibitem{Nov75}
A.A. Novikov.
\newblock Discontinuous martingales.
\newblock {\em Teor. Verojatnost. i Primemen.}, 20:13--28, 1975.

\bibitem{Os12}
A.~Os{\c{e}}kowski.
\newblock A note on the {B}urkholder-{R}osenthal inequality.
\newblock {\em Bull. Pol. Acad. Sci. Math.}, 60(2):177--185, 2012.

\bibitem{Osekobook}
A.~Os{\c{e}}kowski.
\newblock {\em Sharp martingale and semimartingale inequalities}, volume~72 of
  {\em Instytut Matematyczny Polskiej Akademii Nauk. Monografie Matematyczne
  (New Series)}.
\newblock Birkh\"auser/Springer Basel AG, Basel, 2012.

\bibitem{RY}
D.~Revuz and M.~Yor.
\newblock {\em Continuous martingales and {B}rownian motion}, volume 293 of
  {\em Grundlehren der Mathematischen Wissenschaften}.
\newblock Springer-Verlag, Berlin, third edition, 1999.

\bibitem{Rubio86}
J.L. Rubio~de Francia.
\newblock Martingale and integral transforms of {B}anach space valued
  functions.
\newblock In {\em Probability and {B}anach spaces ({Z}aragoza, 1985)}, volume
  1221 of {\em Lecture Notes in Math.}, pages 195--222. Springer, Berlin, 1986.

\bibitem{VerYar}
M.C. Veraar and I.S. Yaroslavtsev.
\newblock Cylindrical continuous martingales and stochastic integration in
  infinite dimensions.
\newblock {\em Electron. J. Probab.}, 21:Paper No. 59, 53, 2016.

\bibitem{Yarodecom}
I.S. Yaroslavtsev.
\newblock Martingale decompositions and weak differential subordination in
  {UMD} {B}anach spaces.
\newblock {\em arXiv:1706.01731, to appear in Bernoulli}, 2017.

\bibitem{Yar17GMY}
I.S. Yaroslavtsev.
\newblock On the martingale decompositions of {G}undy, {M}eyer, and {Y}oeurp in
  infinite dimensions.
\newblock {\em arXiv:1712.00401, to appear in Ann. Inst. Henri Poincar\'{e}
  Probab. Stat.}, 2017.

\bibitem{Y18BDG}
I.S. Yaroslavtsev.
\newblock Burkholder--{D}avis--{G}undy inequalities in {UMD} {B}anach spaces.
\newblock {\em arXiv:1807.05573}, 2018.

\bibitem{Yoe76}
Ch. Yoeurp.
\newblock D\'ecompositions des martingales locales et formules exponentielles.
\newblock pages 432--480. Lecture Notes in Math., Vol. 511, 1976.

\bibitem{Zaa67}
A.C. Zaanen.
\newblock {\em Integration}.
\newblock North-Holland Publishing Co., Amsterdam; Interscience Publishers John
  Wiley \& Sons, Inc., New York, 1967.
\newblock Completely revised edition of An introduction to the theory of
  integration.

\end{thebibliography}

\def\cprime{$'$}

\end{document}